\UseRawInputEncoding
\documentclass[12pt, reqno]{amsart}
\usepackage[margin=1in]{geometry}
\usepackage[mathscr]{eucal}
\usepackage{bm}
\usepackage{mathptmx}
\usepackage{amssymb}
\usepackage{amsthm}
\usepackage{dcolumn}
\usepackage[all]{xy}
\usepackage{enumitem}
\usepackage{pdfrender, xcolor}
\usepackage[hidelinks]{hyperref}

\def\textmatrix#1&#2\\#3&#4\\{\bigl({#1 \atop #3}\ {#2 \atop #4}\bigr)}
\def\dispmatrix#1&#2\\#3&#4\\{\left({#1 \atop #3}\ {#2 \atop #4}\right)}

\newcommand{\beg}{\begin{equation}}
        \newcommand{\eeg}{\end{equation}}
\newcommand{\ben}{\begin{eqnarray*}}
        \newcommand{\een}{\end{eqnarray*}}

\newtheorem{thm}{Theorem}[section]
\newtheorem{cor}[thm]{Corollary}
\newtheorem{lem}[thm]{Lemma}

\newtheorem{prop}[thm]{Proposition}
\numberwithin{equation}{section} \theoremstyle{definition}
\newtheorem{defn}[thm]{Definition}

\newtheorem{eg}[thm]{Example}
\def\textmatrix#1&#2\\#3&#4\\{\bigl({#1 \atop #3}\ {#2 \atop #4}\bigr)}
\def\dispmatrix#1&#2\\#3&#4\\{\left({#1 \atop #3}\ {#2 \atop #4}\right)}
\newcommand{\C}{\mathbb{C}}
\newcommand{\D}{\mathbb{D}}
\newcommand{\N}{\mathbb{N}}
\newcommand{\R}{\mathbb{R}}
\newcommand{\T}{\mathbb{T}}
\newcommand{\BL}{\mathbb{L}}
\newcommand{\X}{\mathbb{X}}
\newcommand{\Y}{\mathbb{Y}}

\newcommand{\HS}{\mathcal{H}}
\newcommand{\KS}{\mathcal{K}}
\newcommand{\Aut}{\mathrm{Aut}}
\newcommand{\Int}{\mathrm{Int}}

\title[Spectral set, complete spectral set and dilation for Banach space operators]{Spectral set, complete spectral set and dilation for Banach space operators}

\author[Jana and Pal]{Swapan Jana and Sourav Pal}

\address[Swapan Jana]{Mathematics Department, Indian Institute of Technology Bombay,
		Powai, Mumbai - 400076, India.} \email{ swapan.jana@iitb.ac.in , swapan.math2015@gmail.com}

\address[Sourav Pal]{Mathematics Department, Indian Institute of Technology Bombay,
		Powai, Mumbai - 400076, India.} \email{souravpal@iitb.ac.in , souravmaths@gmail.com}

	

\begin{document}

\begin{abstract}
Famous results due to von Neumann, Sz.-Nagy and Arveson assert that the following four statements are equivalent; a Hilbert space operator $T$ is a contraction; the closed unit disk $\overline{\mathbb D}$ is a spectral set for $T$; $T$ can be dilated to a Hilbert space isometry; $\overline{\mathbb D}$ is a complete spectral set for $T$. In this article, we show by counter examples that no two of them are equivalent for Banach space operators. If $\mathcal F_r$ is the family of all complex Banach space operators having norm less than or equal to $r$ and if $D_R$ denotes the  open disk in the complex plane with centre at the origin and radius $R$, then we prove by an application of Bohr's theorem that $\overline{D}_R$ is the minimal spectral set for $\mathcal F_r$ if and only if $r=R \slash {3}$. Then we prove the equivalence of the following two facts: the Bohr radius of $D_R$ is $R\slash 3$ and $\sup \{ r>0\,:\, \overline{D}_R \text{ is a spectral set for } \mathcal F_r \}=R \slash 3$. For a complex Banach space $\mathbb X$, we show that the following statements are equivalent: $(i)$ $\mathbb X$ is a Hilbert space; $(ii)$ $\overline{\mathbb D}$ is a spectral set for the forward shift operator $M_z$ on $\ell_2(\mathbb X)$; $(iii)$ $\overline{\mathbb D}$ is a spectral set for the backward shift operator $\widehat{M}_z$ on $\ell_2(\X)$; $(iv)$ $\overline{\mathbb D}$ is a spectral set for every strict contraction on $\mathbb X$; $(v)$ $\overline{\mathbb D}$ is a complete spectral set for every contraction $T$ on $\mathbb X$ with $\|T\|=1$; $(vi)$ $\overline{\mathbb D}$ is a complete spectral set of the identity operator $I_\X$ on $\X$.

\end{abstract}

\maketitle


\section{Introduction}\label{Sec:01}
\vspace{0.2cm}

\noindent
We begin with a few basic definitions, notations and terminologies that we shall follow throughout the paper. We denote by $\D, \, \overline{\D}$ and $\T$, the open unit disk, the closed unit disk and the unit circle respectively in the complex plane $\C$. The open disk with radius $r$ that has centre at the origin or at $\alpha$ is denoted by $D_r$ or $D_r(\alpha)$ respectively. For a domain $\Omega\subset \C$, $H^\infty(\Omega)$ consists of all bounded complex-valued holomorphic functions on $\Omega$. All operators in this article are bounded linear operators acting on Banach or Hilbert spaces over the field of complex numbers. An operator $T$ is said to be a \textit{contraction} or \textit{strict contraction} if $\|T\| \leq 1$ or $\|T\| < 1$ respectively. The algebra of all bounded operators on a Banach space $\X$ is denoted by $\mathcal{B}(\X)$. An \textit{isometry} $V$ between Banach spaces $\X$ and $\Y$ is a linear map $V: \X \to \Y$  satisfying $\|V(x)\| =\|x\|$ for all $x\in \X$. A \textit{unitary} between two Banach spaces is a surjective isometry. For a Banach space $\X$ and $p\in [1, \infty)$, the space $\ell_p(\X)$ consists of all sequences $\{x_n\}$ in $\X$ such that $\sum_{n=1}^\infty \|x_n\|^p < \infty$ and norm of $\{x_n\}$ is defined as $\|\{x_n\}\| = \{ \sum_{n=1}^\infty \|x_n\|^p \}^{\frac{1}{p}}$. Also, for any $n\geq 1$, the space $\ell_p^n(\X) := \underbrace{\X\oplus_p \X \oplus_p \cdots \oplus_p \X}_{n-times}$ consists of all $n$-tuples $(x_1, \cdots , x_n)$ with $x_i \in \X$, equipped with the norm $\|(x_1, \cdots, x_n)\| = \left(\sum_{i=1}^n \|x_i\|^p\right)^{1\slash p}$. The \textit{forward-shift operator} and the \textit{backward-shift operator} on $\ell_p(\X)$ are denoted by $M_z$ and $\widehat{M}_z$ respectively. When $\X$ is a Hilbert space, the operator $\widehat{M}_z$ on $\ell_2(\X)$ coincides with the adjoint of $M_z$. The spectrum of an operator $T$ is denoted by $\sigma(T)$.

\medskip

Following the terminologies due to von Neumann \cite{JVNII} and Arveson \cite{WAII}, we say that a compact set $K \subset \C$ is a \textit{spectral set} for a Banach space operator $T$ if $\sigma(T) \subseteq K$ and von Neumann's inequality holds on $K$, i.e.
\begin{equation} \label{eqn:001A}
\|f(T)\| \leq \|f\|_{\infty, \, K}= \sup_{z\in K}\ \, |f(z)| 
\end{equation}
for all rational functions $f=p /q$, where $p,q \in \C[z]$ with $q$ having no zeros in $K$. Note that $f(T)=p(T)q(T)^{-1}$, where $q(T)$ is invertible as $\sigma(T) \subseteq K$ and $q$ has no zeros in $K$. The algebra of all rational functions with poles off $K$ is called the \textit{rational algebra} over $K$ and is denoted by $Rat(K)$. For each $n\geq 1$, let $Rat_n(K)$ be the algebra of all $n \times n$ matrices with entries in $Rat(K)$. An element $F=[f_{ij}]_{n \times n}$ in $ Rat_n(K)$ is called a \textit{matricial rational function} of order $n$ over $K$ and norm of $F$ is defined by
\[
\|F\|_{\infty, K}= \sup_{z\in K} \, \|[f_{ij}(z)]_{n \times n}\|,
\]
where $\|[f_{ij}(z)]_{n \times n}\|$ is the usual operator norm of the scalar matrix $[f_{ij}(z)]_{n \times n}$. For an operator $T$ acting on a Banach space $\X$ and for $F \in Rat_n(K)$, $F(T)$ is defined to be the operator $F(T)= [f_{ij}(T)]_{n \times n}$, where $f_{ij}(T)$ is defined by (\ref{eqn:001A}) and following Pisier's terminology as in \cite{GPII}, the block matrix of operators $[f_{ij}(T)]_{n \times n}$ is assumed to act on the Banach space $\ell_2^n(\X)$. A compact set $K \subset \C$ is a \textit{complete spectral set} for an operator $T$ on a Banach space $\X$, if $\sigma(T) \subseteq K$ and matricial von Neumann's inequality holds on $K$, i.e.
\[
\|F(T)\| \leq \|F\|_{\infty, K}
\]
for all $F \in Rat_n(K)$ and for all $n \geq 1$. Needless to mention that if $K$ is a complete spectral set for $T$ then it is a spectral set for $T$. We now define spectral set and minimal spectral set for a family of operators which will be useful in the context of this paper.

\begin{defn}
Let $\mathcal F$ be a family of Banach space operators. A compact set $K \subset \C$ is called a \textit{spectral set} for the family $\mathcal F$ if $K$ is a spectral set for every operator in $\mathcal F$. Moreover, $K$ is said to be the \textit{minimal spectral set} for $\mathcal F$ if $K$ is a spectral set for $\mathcal F$ and for any proper compact subset $\widehat{K}$ of $K$, there is an operator $T \in \mathcal F$ such that $\widehat{K}$ is not a spectral set for $T$.
\end{defn}

We are now going to define isometric dilation of a Banach space contraction and this requires a brief motivation for the sake of readers' understanding how the definition of Banach space dilation arrives from that of Hilbert space dilation. Suppose $T$ is a contraction on a Hilbert space $\mathcal{H}.$ An isometry $V$ on a Hilbert space $\mathcal{K}$ is said to be an \textit{isometric dilation} of $T$ if there exists an isometry $W:\mathcal{H}\to \mathcal{K}$ such that
\begin{equation} \label{eqn:001}
f(T)=W^*f(V)W \ \ \text{ or equivalently } \ \ f(\widehat{T})= P_{_{W(\HS)}} f(V)|_{W(\HS)}
\end{equation}
for all $f\in Rat (\overline{\D})$, where $\widehat{T}$ is defined by $\widehat{T}:= \widehat{W} T \widehat{W}^{-1}: W(\HS) \to W(\HS)$ with $\widehat{W}$ being the unitary $\widehat{W}:= W:\HS \to W(\HS)$ and $P_{_{W(\HS)}}$ being the orthogonal projection of $\KS$ onto $W(\HS)$. Note that the dilation space $\KS$ is isomorphic with the orthogonal direct sum $W(\HS) \oplus_2 \BL$, where $\BL$ is the orthogonal complement of $W(\HS)$ in $\KS$. The first obstruction in the way of defining dilation for a Banach space contraction is that the isometry $W$ does not have an adjoint in Banach space setting. In view of (\ref{eqn:001}) it is evident that isometric dilation in Hilbert space can be defined in an equivalent way by avoiding the Hilbert space adjoint of $W$. If we want to adopt this definition for dilation in Banach space, then the next issue is that a Banach space projection can have norm strictly greater than $1$ which can cause imbalance in norm in either side of (\ref{eqn:001}). However, in case of Banach space if the bigger space $\KS$ is chosen to be isomorphic with $W(\HS) \oplus_2 \BL$ for some Banach space $\BL$, then the projection of $\KS$ onto $W(\HS)$ as in (\ref{eqn:001}) becomes a norm-one projection and the norm-issue is resolved. Taking cue from these facts, dilation in Banach space was defined in \cite{JPR} in the following way which generalizes the definition of Hilbert space dilation.

\begin{defn}\label{Definition of Dilation of contractions on Banach space}
Suppose $T$ is a contraction acting on a Banach space $\mathbb{X}$. An isometry $V$ on a Banach space $\widetilde{\mathbb{X}}$ is said to be an \textit{isometric dilation} of $T$ if there is an isometry $W:\mathbb{X}\to \widetilde{\mathbb{X}}$ and a closed linear subspace $\BL$ of $\widetilde{\X}$ such that $\widetilde{\mathbb{X}}$ is isomorphic with $ W(\mathbb{X}) \oplus_2 \BL$ and the operator $\widehat{T}:= \widehat{W} T \widehat{W}^{-1}: W(\X) \to W(\X)$ satisfies
\[
f(\widehat{T})= P_{_{W(\mathbb{X})}} f(V)|_{W(\X)}
\]
for all rational functions $f \in Rat (\overline{\D})$, where $\widehat{W}$ is the unitary (i.e. surjective isometry) $\widehat{W}:= W:\mathbb{X}\to W(\mathbb{X})$ and $P_{_{W(\mathbb{X})}}$ is the norm-one projection of $\widetilde{\X}$ onto $W(\mathbb{X})$.
\end{defn}

These three notions, spectral set, complete spectral set and dilation go hand in hand for a Hilbert space contraction, the classic text \cite{VIP} due to Paulsen provides a detailed study on this. In \cite{JVNII}, von Neumann profoundly found the equivalence of the two facts: a Hilbert space operator $T$ is a contraction and the closed unit disk $\overline{\D}$ is a spectral set for $T$. An appealing next step is Sz.-Nagy's famous dilation theorem \cite{BSN} which states that a Hilbert space operator is a contraction if and only if it dilates to an isometry. Later, Arveson's path-breaking result \cite{WAII} established the equivalence of complete spectral set and dilation of a Hilbert space operator. Thus, combining all these revolutionary results together we have the following theorem which is the primary source of motivation for us.

\begin{thm}\label{thm:401}
For a Hilbert space operator $T$ the following are equivalent:

\begin{enumerate}
\item[(i)] $T$ is a contraction;
\item[(ii)] $T$ dilates to an isometry;
\item[(iii)] $\overline{\D}$ is spectral set of $T$;
\item[(iv)] $\overline{\D}$ is complete spectral set of $T$.
\end{enumerate}
\end{thm}

In this paper, we show by counter examples that no two of the conditions of Theorem \ref{thm:401} are equivalent in general for a Banach space operator $T$. It is evident from of conditions (i) \& (iii) that the operator-norm versus spectral set is an interesting interplay for Hilbert space operators. However, this is no longer true for Banach space operators as we show in Section \ref{Sec:03} via the following theorem.
\begin{thm} \label{thm:new-011}
Let $\mathcal{F}_r$ be the set of all Banach space operators with norm less than or equal to $r$. Then the closed disk $\overline{D}_R=\{ z\in \C \, : \, |z|\leq R \}$ is the minimal spectral set for $\mathcal{F}_r$ if and only if $r= \frac{R}{3}$.
\end{thm}
This is Theorem \ref{thm:302} in this paper and is proved using Bohr's theorem. Also, this generalizes a previous result from \cite{VEK & VIM} due to Kacnel\'{s}on and Macaev, see Theorem \ref{thm:301} of this paper. The fact that no two of the other conditions of Theorem \ref{thm:401} are equivalent for Banach space operators is established in Section \ref{Subsec:041} by different counter examples. Interestingly, Theorem \ref{thm:new-011} sets a connection between these two fundamental notions--  spectral set for Banach space operators and Bohr radius of a disk. Indeed, in Theorem \ref{thm:303} we prove the equivalence of the following two facts: the Bohr radius of $D_R$ is $R\slash 3$ and $\sup \{ r>0\,:\, \overline{D}_R \text{ is a spectral set for } \mathcal F_r \}=R \slash 3$.

\smallskip

While investigating the validity of Theorem \ref{thm:401} for Banach space operators, Foias \cite{CFI} achieved a very interesting result which states the following: every contraction $T$ on a complex Banach space $\X$ with $\|T\|=1$ has $\overline{\D}$ as a spectral set if and only if $\X$ is a Hilbert space. In \cite{JPR}, the authors of this article and Roy found a dilation theoretic characterization of a Hilbert space. More precisely, it is proved in \cite{JPR} that a complex Banach space $\X$ is a Hilbert space if and only if every strict contraction $T \in \mathcal B(\X)$ dilates to an isometry if and only if the function $A_T(x)=(\|x\|^2-\|Tx\|^2)^{\frac{1}{2}}$ defines a norm on $\X$ for every strict contraction $T \in \mathcal B(\X)$. Another contribution of this article is to add to the account a few new conditions each of which is necessary and sufficient for a Banach space to become a Hilbert space as shown below.

\begin{thm}
Let $\X$ be a complex Banach space. Then the following are equivalent:

\begin{enumerate}
\item[(i)] $\X$ is a Hilbert space;

\smallskip

\item[(ii)] $\overline{\D}$ is a spectral set for $\widehat{M}_z$ on $\ell_2(\X)$;

\smallskip

\item[(iii)] $\overline{\D}$ is a spectral set for $M_z$ on $\ell_2(\X)$;

\smallskip

\item[(iv)] $\overline{\D}$ is a spectral set for every strict contraction on $\X$;

\smallskip

\item[(v)] $\overline{\D}$ is a complete spectral set for every contraction $T$ on $\X$ with $\|T\|=1$;

\smallskip

\item[(vi)] $\overline{\D}$ is a complete spectral set for the identity operator $I_\X$ on $\X$.
\end{enumerate}
\end{thm}
This is Theorem \ref{thm:103} in this paper.

\smallskip

\section{Bohr's theorem and norm vs. spectral set for a Banach space operator}\label{Sec:03}

\vspace{0.2cm}

\noindent Recall that $H^{\infty}(\D)$ denotes the algebra of all bounded complex-valued holomorphic functions defined on the unit disk $\D$. In 1914, H. Bohr \cite{HB} proved the following remarkable result whose impact is extraordinary till date.
\begin{thm} [H. Bohr, \cite{HB}]
For any $f\in H^{\infty}(\D)$ if $f(z)=\sum_{n=0}^{\infty} a_nz^n$, then
\begin{equation} \label{eq:301}
\sum_{n=0}^{\infty} |a_nz^n| \leq \|f\|_{\infty,\,  \overline{\D}} \quad \text{ for } \, \, \,  |z| \leq \frac{1}{3}.
\end{equation}
Moreover, the bound $\dfrac{1}{3}$ is the sharpest.
\end{thm}
Actually, Bohr's primary bound for this result was $1 \slash 6$ which was further improved to $1\slash 3$ by  M. Riesz, I. Schur, and N. Wiener, independently. A few years later, S. Sidon \cite{Sidon} gave a different proof to Bohr's theorem, which was subsequently rediscovered by M. Tomi\'c in \cite{Tomic}. The bound $1\slash 3$ is the sharpest in the sense that for any $\alpha\in \D$ with $|\alpha|> 1 \slash 3$, there exists $g\in H^\infty(\D)$ satisfying $\sum_{n=0}^\infty |\frac{g^{(n)}(0)}{n!}\alpha^n| > \|g\|_{\infty, \, \overline{\D}}$, where $g^{(n)}(0)$ is the $n$-th derivative of $g$ at $0$. The inequality (\ref{eq:301}) is known as \textit{Bohr inequality} and the sharpest bound $1\slash 3$ is called the \textit{Bohr radius} of $\D$ with respect to the algebra $H^{\infty}(\D)$. Study of Bohr inequality and Bohr radius for holomorphic functions on domains in several variables has been  an active area of research, e.g., see \cite{LA, OB, HPB & DK, AD & LF, DFOOS} and the references therein.

\smallskip

Theorem \ref{thm:401} shows that $\overline{\D}$ is a spectral set for a Hilbert space operator if and only if it is a contraction. In \cite{VEK & VIM}, V. E. Kacnel\'{s}on and V. I. Macaev found the following fascinating result which shows that the operator-norm versus spectral set issue is not settled in the exact same way for Banach space operators.
\begin{thm}[Theorem 2, \cite{VEK & VIM}]\label{thm:301}
The closed disk $\overline{D}_\rho$ with centre at the origin and radius $\rho$, is a spectral set of all Banach space contractions if and only if $\rho \geq 3$.
\end{thm}
The theorem of Kacnel\'{s}on and Macaev indicates a possible connection between Bohr radius and spectral set of Banach space operators which we explore and determine conclusively in this Section. We begin with the following generalization of Bohr's theorem for a disk $D_R$ with centre at the origin and radius $R$.

\begin{thm}[Generalized Bohr's Theorem]  \label{prop:301}
Let $D_R=\{z\in \C \, : \, |z|<R\}$. Then the Bohr radius of $D_R$ with respect to $H^\infty(D_R)$ is $\frac{R}{3}$.
\end{thm}

\begin{proof}
Let us consider the biholomorphic map $\phi: \D \rightarrow D_R$ defined by $\phi(\xi)=R\xi$ for all $\xi$ in the unit disk $\D$. Then for any $f \in H^{\infty}(D_R)$, the function $g=f_{\circ}\phi \in H^{\infty}(\D)$ and $\|f\|_{\infty, \, \overline{D}_R}=\|g\|_{\infty, \, \overline{\D}}$. So, if $f(\xi)=\sum_{n=0}^{\infty} \, a_n\xi^n$, then $g(\xi)=f(\phi(\xi))=\sum_{n=0}^{\infty} \, a_nR^n\xi^n$. Since Bohr radius of $\D$ is $1 \slash 3$, it follows that $\sum_{n=0}^{\infty} \, |a_n||R\xi|^n \leq \|g\|_{\infty, \, \overline{\D}}$ for all $\xi$ satisfying $|\xi|\leq 1 \slash 3$. Substituting $z=R\xi$, we have that $\sum_{n=0}^{\infty} \, |a_n||z|^n \leq \|g\|_{\infty, \, \overline{\D}}=\|f\|_{\infty, \, \overline{D}_R}$ for all $z=R\xi$ satisfying $|\xi|\leq 1\slash 3$, i.e., for all $z$ satisfying $|z|\leq R\slash 3$. Also, for any $\beta$ satisfying $|\beta|> R \slash 3$, we have $|\beta \slash R|> 1 \slash 3$ and hence there is a function $g_1 \in H^{\infty}(\D)$ with $g_1(\xi)=\sum_{n=0}^{\infty} \, b_n\xi^n$ such that
\begin{equation} \label{eqn:new-0221}
\sum_{n=0}^{\infty} \, |b_n||\beta \slash R|^n > \|g_1\|_{\infty, \, \overline{\D}}.
\end{equation}
Consider the function $f_1(\xi)=g_1( \xi \slash R)=\sum_{n=0}^{\infty} \, b_n(\xi \slash R)^n$ which is in $H^{\infty}(D_R)$. Then $\|g_1\|_{\infty, \, \overline{\D}}=\|f_1\|_{\infty, \, \overline{D}_R}$ and at the point $\xi = \beta$, we have $f_1(\beta)=g_1( \beta \slash R)=\sum_{n=0}^{\infty} \, b_n(\beta \slash R)^n$. It follows from (\ref{eqn:new-0221}) that $\sum_{n=0}^{\infty} \, b_n|\beta \slash R|^n > \|g_1\|_{\infty, \, \overline{\D}}=\|f_1\|_{\infty, \, \overline{D}_R}$. This shows that the Bohr radius of $D_R$ with respect to $H^{\infty}(D_R)$ is $R\slash 3$.
\end{proof}

For a compact set $K\subset \C$, consider the algebra $\mathcal{A}(K)$ consisting of complex-valued functions that are holomorphic in the interior of $K$ and are continuous on $K$, i.e.,
\begin{equation}\label{eq:303A}
\mathcal{A}(K) = \left\{f: K \to \C ~|~ \mbox{ $f$ is holomorphic in $\Int(K)$ and continuous on $K$}\right\},
\end{equation}
equipped with the norm $\|f\|_{\infty, K} = \sup\{|f(z)|: z\in K\}$. The operator $T_K: \mathcal{A}(K) \to \mathcal{A}(K)$ defined by $T_K(f)(z) = zf(z)$ is of type forward shift with 
\begin{equation}\label{eq:303}
\|T_K\| = \sup\{|z|: z\in K\}.
\end{equation}
If the complement of $K$ is connected and if $K$ contains a neighborhood of the origin, it was mentioned without proof by Crownover in \cite{RMC} that the spectrum of the operator $T_K$ coincides with the set $K$. In fact, the result is true in general for any compact set $K$. We could not locate a proof to this result in general setting anywhere in the literature and thus we give a proof here.
\begin{prop}\label{prop:302}
Let $K\subset \C$ be a compact set and let $\mathcal A(K)$ be as in $(\ref{eq:303A})$. Then the spectrum of the operator $T_K:\mathcal{A}(K) \to \mathcal{A}(K)$ defined by $T_K(f)(z)= zf(z)$, is equal to $K$, that is $\sigma(T_K) = K$.
\end{prop}

\begin{proof}
Let $\lambda\in K$ be arbitrary. For any $f\in \mathcal{A}(K)$ we have that $(T_K- \lambda I)(f)(\lambda) = \lambda f(\lambda) - \lambda f(\lambda) = 0$ and this shows that if $g \in Ran(T_K - \lambda I)$, then $g(\lambda)=0$. Evidently, not all functions in $\mathcal A(K)$ vanish at $\lambda$ and consequently $(T_K- \lambda I)$ is not surjective. Thus, $\lambda\in \sigma(T_K)$ and hence $K \subseteq \sigma(T_K)$. We now show the other way, that is $\sigma(T_K) \subseteq K$. It suffices to show the contrapositive form, that is $\lambda \notin K$ implies that $\lambda \notin \sigma(T)$. Let $\lambda\notin K$. Since $K$ is compact, there exists $\delta >0$ such that $D_\delta(\lambda) \cap K = \emptyset$. Consider the function $\phi: K \to K $ defined by $\phi(z) = \frac{\delta}{z-\lambda}$. Then $\phi \in \mathcal{A}(K)$ and since $|z-\lambda| \geq \delta$ for all $z\in K$, we have that $\|\phi\|_{\infty,\, K} \leq 1$. The function $S_\phi: \mathcal{A}(K) \to \mathcal{A}(K)$ defined by $S_\phi(f)(z) = \phi(z)f(z)$ is a bounded operator with $\|S_\phi\| = \|\phi \|_{\infty, K} \leq 1$. Also, for any $f\in \mathcal{A}(K)$ and for $z\in K$ we have
\[
(T_K - \lambda I)S_\phi (f)(z) = \delta f(z)= S_\phi (T_K - \lambda I)(f)(z),
\]
which shows that $(T_K - \lambda I)$ is invertible with inverse $S_\phi \slash \delta$. So, $\lambda\notin \sigma(T_K)$. Consequently, we have $\sigma(T_K) \subseteq K$. This completes the proof.
\end{proof}

\begin{lem}\label{lem:new301}
Let $\X$ be a Banach space and $T\in \mathcal{B}(\X)$ with $\|T\| \leq \frac{R}{3}$ for some $R > 0$. Suppose $K = \overline{D}_R \setminus D_\delta(\alpha)$, where $\frac{R}{3}< |\alpha| < R$ and $\delta >0$ be such that $D_\delta(\alpha)=\{ z\in D_R : |z -\alpha| < \delta \} \subset \{z\in {D}_R: \frac{R}{3} < |z| < R\}$. Then von Neumann's inequality holds over $K$ for the function $g(z)=(z- \alpha)^{-m}$ for all $m\in \N$, that is,
$
   \|g(T)\| \leq \|g\|_{\infty,\, K}.
$
\end{lem}

\begin{proof}
First we note that $\|g\|_{\infty, \, K} = \frac{1}{\delta^m}$. Indeed, for all $z\in K$, $|z-\alpha| \geq \delta$ and $g(\alpha+ \delta)= \frac{1}{\delta^m}$, along with $(\alpha + \delta)\in K$. Since $\|T\| \leq \frac{R}{3}$, $\sigma(T)\subseteq \overline{D}_{\frac{R}{3}}$. So, $(T-\alpha I_\X)$ is invertible in $\mathcal{B}(\X)$ and its inverse is given by
\[
    (T-\alpha I_\X)^{-1} = \sum_{n=0}^\infty -\frac{T^n}{\alpha^{n+1}}.
\]
Consequently, we have
\begin{align} \label{display:new-01}
   \|g(T)\| = \|(T- \alpha I_\X)^{-m}\| \leq \|(T-\alpha I_\X)^{-1}\|^m \leq \left(\frac{1}{|\alpha|} \sum_{n=0}^\infty \frac{\|T\|^n}{|\alpha|^n} \right)^m & \leq \left(\frac{1}{|\alpha|} \sum_{n=0}^\infty \frac{R^n}{|3\alpha|^n} \right)^m \notag \\
   & = \frac{1}{\left(|\alpha| - \frac{R}{3}\right)^m} \notag \\
   & \leq \frac{1}{\delta^m} = \|g\|_{\infty,\, K},  
\end{align}
where the last inequality follows from the fact that $D_\delta(\alpha) \subset \{z\in {D}_R: \frac{R}{3} < |z| < R\}$.
\end{proof}

Being armed with all necessary results, we are now in a position to present one of the main results of this paper.

\begin{thm}\label{thm:302}
Let $\mathcal{F}_r$ be the set of all Banach space operators with norm less than or equal to $r$. Then the closed disk $\overline{D}_R=\{ z\in \C \, : \, |z|\leq R \}$ is the minimal spectral set for $\mathcal{F}_r$ if and only if $r= \frac{R}{3}$.
\end{thm}

\begin{proof}
We prove the forward implication first. Assume that $\overline{D}_R$ is the  minimal spectral set of $\mathcal{F}_r$. We start by showing that $0\leq r \leq {R}\slash {3}$. If possible let $r> {R}\slash {3}$. Since the Bohr radius of $D_R$ with respect to $H^\infty(D_R)$ is ${R}\slash {3}$ (see Theorem \ref{prop:301}), there is $f\in H^\infty(D_R)$ with $f(z) = \sum_{n=0}^\infty a_nz^n$ such that
\begin{equation} \label{eqn:new-021}
\sum_{n=0}^\infty |a_n|r^n > \|f\|_{\infty,\, \overline{D}_R}.
\end{equation}
Now consider the operator $rM_z$, where $M_z$ is the forward shift operator on $\ell_1(\C)$ that maps an element $(a_0,a_1, \dots) \in \ell_1(\C)$ to $(0, a_0,a_1,\dots)$. The operator $f(rM_z) = \sum_{n=0}^\infty a_nr^n M_z^n$ can be represented by the following matrix: 
\[
    f(rM_z)= \begin{bmatrix}
    a_0 & 0 & 0 & 0&  \cdots \\
    ra_1 & a_0 & 0 & 0 & \cdots \\
    r^2 a_2 & ra_1 & a_0 & 0 & \cdots \\
    \vdots & \vdots & \vdots & \vdots & \ddots \\
    \end{bmatrix}.
\]
Thus, $\|f(rM_z)\| = \sum_{n=0}^\infty |a_n|r^n$. By hypothesis, $\overline{D}_R$ is a spectral set of $rM_z$ as $\|rM_z\|=r$. So, $\|f(rM_z)\| = \sum_{n=0}^\infty |a_n|r^n \leq \|f\|_{\infty,\, \overline{D}_R}$, which contradicts (\ref{eqn:new-021}). Consequently, we have $0\leq r \leq {R}\slash {3}$. Next, we show that $r= {R}\slash {3}$. Let us assume the contrary, that is, $r< {R}\slash {3}$. Let $S=3r$. Then $S < R$. By Theorem \ref{prop:301}, the Bohr radius of $D_S$ with respect to $H^\infty(D_S)$ is ${S}\slash {3}$ which is equal to $r$. Let $f\in H^\infty(D_S)$ be arbitrary with $f(z) = \sum_{n=0}^\infty b_nz^n$, and let $T$ be any operator in $\mathcal{F}_r$. Then 
\[
  \|f(T)\| = \left\|\sum_{n=0}^\infty b_n T^n \right\| \leq \sum_{n=0}^\infty |b_n|~ \|T\|^n \leq \sum_{n=0}^\infty |b_n|r^n \leq \|f\|_{\infty,\, \overline{D}_S},
\]
where the second last inequality holds as $\|T\| \leq r=S \slash{3}$ and the last inequality follows from the fact that the Bohr radius of $D_S$ with respect to $H^{\infty}(D_S)$ is equal to $S\slash 3$. This shows that $\overline{D}_S$ is a spectral set of $T$. Since $T \in \mathcal F_r$ was arbitrary, it follows that $\overline{D}_S$ is a spectral set of $\mathcal{F}_r$. This contradicts the fact that $\overline{D}_R$ is the minimal spectral set for $\mathcal F_r$ as $S<R$. Hence, $r= R \slash 3$.

\smallskip

Now, we prove the sufficiency. Assume $r= \frac{R}{3}$. 
We will prove only the minimality condition in the statement, since the proof of the fact that $\overline{D}_R$ is a spectral set of $\mathcal{F}_r$ is similar to the proof of $r=\frac{R}{3}$ in the necessary part. Suppose $K$ is a spectral set of $\mathcal{F}_r$, where $K\subset \overline{D}_R$ is a proper compact subset. Note that $\overline{D}_{\frac{R}{3}} \subseteq K$. Indeed, $cI_\X\in \mathcal{F}_r$ for all $c\in \overline{D}_r$, and for all Banach space $\X$. So, $K$ is a  spectral set of $cI_\X$, which implies that $\sigma(cI_\X) =\{c\}\subseteq K$. If $D_R \subseteq K \subseteq \overline{D}_R$, then $K = \overline{D}_R$. Therefore, there exists $\alpha \in D_R \setminus K$ and $\delta > 0$ such that $D_\delta(\alpha)=\{z\in D_R: |z-\alpha|< \delta \} \subset D_R$ and $D_\delta(\alpha) \cap K = \emptyset$. Consequently, we have $D_\delta(\alpha) \subset \{ z \in D_R : \frac{R}{3} < |z| < R \}$. Let $\widehat{K} = \overline{D}_R \setminus D_\delta(\alpha)$. Then $\widehat{K}$ is a compact subset of $\overline{D}_R$, and $K \subseteq \widehat{K}$. Since $K$ is a spectral set of $\mathcal{F}_r$, it follows that $\widehat{K}$ is also a spectral set of $\mathcal{F}_r$.

Now we consider the Banach space $\mathcal{A}(\widetilde{K})$ as in (\ref{eq:303A}), where $\widetilde{K}= \overline{D}_{\frac{R}{3}}$ and also consider the bounded operator $T_{\widetilde{K}}: \mathcal{A}(\widetilde{K}) \to \mathcal{A}(\widetilde{K})$ defined by $T_{\widetilde{K}}(f)(z) = zf(z)$ for all $f \in \mathcal{A}(\widetilde{K})$ and $z \in \widetilde{K}$. It follows from Equation-(\ref{eq:303}) and Proposition \ref{prop:302} that $\|T_{\widetilde{K}}\| = \frac{R}{3}$ and that $\sigma(T_{\widetilde{K}})=\widetilde{K}$. So $(T_{\widetilde{K}}- \alpha I)$ is an invertible operator on $\mathcal{A}(\widetilde{K})$, where $I$ is the identity operator on $\mathcal{A}(\widetilde{K})$. For any $(c\in (0,1))$, let us consider the function $g_c: \widehat{K} \to \C$ defined by 
\[
 g_c(z) = \frac{\frac{z}{R} - c}{1- \frac{cz}{R}} + \frac{\frac{\delta}{z- \alpha} - c }{1- \frac{c \delta}{z- \alpha}}, \quad z\in \widehat{K}.
\]
It is evident that $1- \frac{cz}{R} =0$  if and only if $|z| > R$ and $1- \frac{c \delta}{z- \alpha} =0$ if and only if $|z-\alpha| < \delta$. So, we have that $g_c\in Rat(\widehat{K})$ for all $c\in (0,1)$. By maximum modulus principle, each $g_c$ attains its norm on $ A \cup B$, where $A= \{z\in \overline{D}_R: |z|= R\}$, and $B=\{z\in \overline{D}_R: |z-\alpha| = \delta\}$. For each $z\in A$, we have that $|z-\alpha| > \delta$. Therefore, we have
\[
   \left|\frac{\frac{z}{R} - c}{1- \frac{cz}{R}}\right| = 1 \quad \text{ and } \quad \left|\frac{\frac{\delta}{z- \alpha} - c}{1- \frac{c \delta}{z- \alpha}}\right| < 1, \quad z\in A.
\]
Similarly, for all $z\in B$, we have $|z| < R$ which gives
\[
\left|\frac{\frac{z}{R} - c}{1- \frac{cz}{R}}\right| < 1 \quad \text{ and } \quad \left|\frac{\frac{\delta}{z- \alpha} - c }{1- \frac{c \delta}{z- \alpha}}\right| = 1, \quad z\in B.
\]
Consequently, we have $\|g_c\|_{\infty, \widehat{K}} \leq M < 2$ for all $0 < c < 1$. Since $T_{\widetilde{K}}\in \mathcal{F}_{r}$ (as $\|T_{\widetilde{K}}\|=\frac{R}{3}$) and $\widehat{K}$ is a spectral set of $T_{\widetilde{K}}$, we have
\begin{equation}\label{eq:new202}
   \|g_c(T_{\widetilde{K}})\| \leq \|g_c\|_{\infty, \widehat{K}} \leq M < 2, \quad 0 < c < 1.
\end{equation}
For all $c$ satisfying $0< c < 1$, we have $\|\frac{cT_{\widetilde{K}}}{R}\| \leq \frac{c}{3} < 1$. Also, $(T_{\widetilde{K}}- \alpha I)$ is invertible since we have $\|T_{\widetilde{K}}\| \leq \frac{R}{3} < |\alpha|$. Needless to mention that the inequality (\ref{display:new-01}) in the proof of Lemma \ref{lem:new301} implies that $\|c\delta (T_{\widetilde{K}} - \alpha I)^{-1}\|\leq c < 1$ when $m=1$. Therefore, we have
\begin{align}\label{eq:new203}
g_c(T_{\widetilde{K}}) & = \left(\frac{T_{\widetilde{K}}}{R} - c I \right) \left(I - \frac{cT_{\widetilde{K}}}{R}\right)^{-1} + \left(\delta(T_{\widetilde{K}}- \alpha I)^{-1} - cI\right)\left(I - c\delta (T_{\widetilde{K}} - \alpha I)^{-1}\right)^{-1} \nonumber \\
& = \left(\frac{T_{\widetilde{K}}}{R} - c I \right) \left(\sum_{n=0}^\infty \frac{c^n T_{\widetilde{K}}^n}{R^n}\right) + \left(\delta(T_{\widetilde{K}}- \alpha I)^{-1} - cI\right)\left(\sum_{n=0}^\infty c^n \delta^n (T_{\widetilde{K}}- \alpha I)^{-n}\right) \nonumber \\
& = \left[-cI + \sum_{n=1}^\infty \left(\frac{c^{n-1}}{R^n} - \frac{c^{n+1}}{R^n}\right)T_{\widetilde{K}}^n \right] + \left[ -cI + \sum_{n=1}^\infty \left(\delta^n c^{n-1} - \delta^n c^{n+1}\right)\left(T_{\widetilde{K}}- \alpha I\right)^{-n} \right] \nonumber \\
& = -2cI + (1-c^2)\sum_{n=1}^\infty \frac{c^{n-1}}{R^n}T_{\widetilde{K}}^n + (1-c^2)\sum_{n=1}^\infty \delta^n c^{n-1}\left(T_{\widetilde{K}}- \alpha I\right)^{-n},
\end{align}
where the convergence of the second series in \eqref{eq:new203} follows from the fact that $\|(T_{\widetilde{K}}-\alpha  I)^{-n}\| \leq \frac{1}{\delta^n}$ (again see (\ref{display:new-01}) in the proof of Lemma \ref{lem:new301}), for all $n\in \N$. It follows from the definition of $T_{\widetilde{K}}$ that 
\[
(T_{\widetilde{K}}-\alpha I)^{-n}(\mathbf{1})(z) = (z-\alpha)^{-n}, \quad z\in \widetilde{K}=\overline{D}_{\frac{R}{3}},
\]
where $\mathbf{1}$ is the constant function that takes the value $1$. Thus, we have from \eqref{eq:new203} that
\[
g_c(T_{\widetilde{K}})(\mathbf{1})(z) = -2c + (1-c^2)\sum_{n=1}^\infty \frac{c^{n-1}}{R^n}z^n + (1-c^2) \sum_{n=1}^\infty \delta^n c^{n-1}(z-\alpha)^{-n}, \quad z\in \widetilde{K}=\overline{D}_{\frac{R}{3}}.
\]

Consequently, for all $0 < c < 1$, we have
\begin{align}\label{eq:new204}
 \|g_c(T_{\widetilde{K}})\| \geq \|g_c(T_{\widetilde{K}})(\mathbf{1})\|_{\infty, \overline{D}_{\frac{R}{3}}} \geq \left| g_c(T_{\widetilde{K}})(\mathbf{1})(0)\right| = \left|-2c + (1-c^2) \sum_{n=1}^\infty (-1)^n\frac{\delta^n c^{n-1}}{\alpha^n} \right|
\end{align}
Since, $D_\delta(\alpha)\cap K = \emptyset$, we have $\delta \leq |\alpha| - \frac{R}{3} < |\alpha|$. Thus, it follows from \eqref{eq:new204} and \eqref{eq:new202} that
\begin{align}\label{eq:new205}
   M\geq \|g_c(T_{\widetilde{K}})\| \geq \left|2c + \frac{(1-c^2)\delta}{\alpha} \frac{1}{1+ \frac{\delta c}{\alpha}}\right| = \left| 2c + \frac{(1-c^2)\delta}{\alpha + \delta c}\right|, \quad 0 < c < 1.
\end{align}
Now, taking limit $c \to 1$ in \eqref{eq:new205}, we have $M \geq 2$ which contradicts \eqref{eq:new202}. Consequently, $\widehat{K}$ and hence $K$ is not a spectral set for $\mathcal{F}_r$. Therefore, no proper compact subset of $\overline{D}_{R}$ is a spectral set of $\mathcal{F}_r$. Hence, $\overline{D}_R$ is the minimal spectral set for $\mathcal{F}_r$, where $r= \frac{R}{3}$ and this completes the proof.
\end{proof}

The following is an obvious corollary of the preceding theorem.

\begin{cor} \label{cor:new-021}
Let $\mathcal F_r$ be the set of all Banach space operators with norm less than or equal to $r$. Then the closed unit disk $\overline{\D}$ is the minimal spectral set for $\mathcal F_r$ if and only if $r = \frac{1}{3}.$
\end{cor}

Theorem \ref{thm:302} sets a connection between these two important classical concepts-- the Bohr radius of a disk and spectral set for a family of Banach space operators with certain bound on their norms. We conclude this Section by specifying this connection explicitly in the following theorem.
\begin{thm}\label{thm:303}
Let $\mathcal F_r$ be the set of all Banach space contractions having norms less than or equal to $r$ and let $D_R=\{ z\in \C \,:\, |z|<R \}$. Then the following are equivalent:

\begin{enumerate}
\item[(i)]
the Bohr radius of $D_R$ with respect to the algebra $H^\infty(D_R)$ is $\frac{R}{3}$;

\smallskip

\item[(ii)] $\sup\left\{r>0:\mbox{$\overline{D}_R$ is a spectral set of $\mathcal{F}_r$}\right\} = \frac{R}{3}$;

\smallskip

\item[(iii)] $\inf\left\{r>0:\mbox{$\overline{D}_R$ is not a spectral set of $\mathcal{F}_r$}\right\} = \frac{R}{3}$.
\end{enumerate}
\end{thm}

\begin{proof}
We prove $(i)\Leftrightarrow (ii)$ and $(i)\Leftrightarrow (iii)$.

\smallskip

\noindent $(i)\Rightarrow (ii)$. Denote by $S$ the following set: $S= \left\{r>0:\mbox{$\overline{D}_R$ is a spectral set of $\mathcal{F}_r$}\right\}$. It follows from the proof of the necessity part of Theorem \ref{thm:302} that ${R}\slash {3}$ is an upper bound of $S$. So, $\sup( S)$ exists in $\R$ and $\sup(S) \leq {R}\slash {3}$. If possible let, $\sup(S) < {R} \slash {3}$. Choose $t \in (\sup(S) ,{R}\slash {3})$. Then for any $T\in \mathcal{F}_t$ and for any $f\in H^\infty(D_R)$, we have 
\begin{equation}\label{eq:304}
  \|f(T)\| \leq \sum_{n=0}^\infty \frac{|f^{(n)}(0)|}{n!} \|T\|^n \leq \sum_{n=0}^\infty |\frac{|f^{(n)}(0)|}{n!}| t^n \leq \|f\|_{\infty, \overline{D}_R}. \quad \left[~ \text{since }\, t< \frac{R}{3}~\right]
\end{equation}
This shows that $\overline{D}_R$ is a spectral set of $\mathcal{F}_t$, which contradicts the choice of $t>\sup(S)$. Consequently, we have $\sup(S) = {R}\slash {3}$.

\smallskip

\noindent $(ii) \Rightarrow (i)$. Let $f\in H^\infty(D_R)$ and $t\leq \frac{R}{3} = \sup(S)$, where $S$ is as in the previous part, i.e., $(i)\Rightarrow (ii)$. Consider the operator $T= tM_z\in \mathcal{F}_t$, where $M_z$ is the forward shift on $\ell_1(\C)$. Then
\[
  \sum_{n=0}^\infty \frac{|f^{(n)}(0)|}{n!} t^n = \|f(tM_z)\| \leq \|f\|_{\infty, \overline{D}_R},
\]
which shows that the Bohr radius of $D_R$ with respect to $H^\infty (D_R)$ is  greater than or equal to ${R}\slash {3}$. If the Bohr radius of $D_R$ is equal to $\alpha$ with $\alpha > {R}\slash {3}$, then we can choose $t$ in the interval $({R}\slash {3},\alpha)$. Following the same argument as in (\ref{eq:304}), we have that $\overline{D}_R$ is a spectral set of $\mathcal{F}_t$, which contradicts the hypothesis that $\sup(S) = {R}\slash {3}$ and also the choice of $t > \sup(S)$.

\smallskip

\noindent $(i)\Rightarrow (iii)$. Denote by $A$ the following set: $A= \left\{r>0:\mbox{$\overline{D}_R$ is not a spectral set of $\mathcal{F}_r$}\right\}$. The proof of $(i)\Rightarrow (ii)$ shows that if $t\leq {R}\slash {3}$, then $\overline{D}_R$ is a spectral set of $\mathcal{F}_t$. So, the contrapositive form tell us that ${R}\slash {3}$ is a lower bound for $A$. Thus, $\inf(A)$ exists in $\R$ and $\inf(A)\geq {R}\slash {3}$. If possible, assume that $\inf(A) > {R}\slash {3}$. Choose $t\in ({R} \slash {3}, \inf(A))$. Since the Bohr radius of $D_R$ with respect to $H^\infty(D_R)$ is ${R} \slash {3}$, there exists $f\in H^\infty(D_R)$ such that
\[
  \sum_{n=0}^\infty \frac{|f^{(n)}(0)|}{n!} t^n = \|f(tM_z)\| > \|f\|_{\infty, \overline{D}_R}.
\]
This shows that $\overline{D}_R$ is not a spectral set of $\mathcal{F}_t$, which contradicts the choice of $t$, i.e., $t < \inf(A)$. Hence, we have that $\inf(A) ={R}\slash {3}$.

\smallskip

\noindent $(iii) \Rightarrow(i)$. First, we show that the Bohr radius of $D_R$ with respect to $H^\infty(D_R)$ is less than or equal to ${R} \slash {3}$. Let $t > {R} \slash {3} = \inf(A)$, where $A$ is as in $(i) \Rightarrow (iii)$. The hypothesis shows the existence of an element $f\in H^\infty(D_R)$ satisfying 
\[
\sum_{n=0}^\infty \frac{|f^{(n)}|(0)}{n!} t^n > \|f\|_{\infty, \overline{D}_R}.
\]
Indeed, if the reverse inequality holds for all $f\in H^\infty(D_R)$, then we see that $\overline{D}_R$ is a spectral set for $\mathcal{F}_s$ for all $s\leq t$. Again, following the contrapositive form we have $\inf(A) \geq t$, which contradicts the choice of $t$ as $t> \inf(A)$. Now, let the Bohr radius of $D_R$ with respect to $H^\infty(D_R)$ be $\alpha$. We show that $\alpha = {R}\slash {3}$. If possible let $\alpha < {R}\slash {3}$. Choose $t$ form the interval $(\alpha,{R}\slash{3})$. The same argument as in $(i) \Rightarrow (iii)$ shows that $\overline{D}_R$ is not a spectral set for $\mathcal{F}_t$. This contradicts the hypothesis that $\inf(A) = {R}\slash{3}$ and the choice of $t$, i.e., $t < \inf(A)$. Consequently, the Bohr radius of $D_R$ with respect to $H^\infty(D_R)$ is ${R}\slash{3}$. The proof is now complete.
\end{proof}

\smallskip

\section{Spectral set and complete spectral set for Banach space operators vs. Hilbert space}\label{Sec:02}

\vspace{0.2cm}

\noindent  For a Banach space operator $T$, the results of Section \ref{Sec:03} show a clear conflict between the two statements: $\|T\|\leq 1$ and $\overline{\D}$ is a spectral set for $T$. Going deep into investigating the relation between spectral set and norm of a Banach space operator, Foias \cite{CFI} found the following surprising result.

\begin{thm}[Foias, \cite{CFI}] \label{thm:Foias-New}
A complex Banach space $\X$ is a Hilbert space if and only if the closed unit disk $\overline{\D}$ is a spectral set for all $T\in \mathcal{B}(\X)$ with $\|T\| =1$.
\end{thm} 
In \cite{JPR}, the authors of this article and S. Roy produced another set of equivalent conditions each of which converts a Banach space into a Hilbert space. We recall a part of that theorem here.

\begin{thm}[\cite{JPR}, Theorem 6.12]\label{thm:101}
Let $\mathbb{X}$ be a complex Banach space. Then the following are equivalent.

\smallskip

\begin{enumerate}

\item[(i)] $\mathbb{X}$ is a Hilbert space.

\smallskip

\item[(ii)] Every strict contraction $T$ on $\X$ dilates to an isometry.

\smallskip

\item[(iii)] For every strict contraction $T\in \mathcal{B}(\mathbb{X}),$ the function $A_T: \mathbb{X}\to [0,\infty)$ given by 
\[
  A_T(x) = \left( \| x\|^2 - \|Tx\|^2 \right)^{\frac{1}{2}}, \quad x\in \mathbb{X},
\]
defines a norm on $\mathbb{X}.$
\end{enumerate}
\end{thm}
 This leaves us with a natural question: what are all strict contractions on a complex non-Hilbert Banach space that admit isometric dilation ? This question was further answered in the same article in the following way.
 
\begin{thm}[\cite{JPR}, Theorem 6.13]\label{thm:102}
Suppose $\X$ is a complex Banach space. Then a strict contraction $T$ on $\X$ dilates to an isometry if and only if the function $A_T: \mathbb{X}\to [0,\infty)$ given by $
  A_T(x) = \left( \| x\|^2 - \|Tx\|^2 \right)^{\frac{1}{2}}$
defines a norm on $\mathbb{X}.$ 
\end{thm}

In this Section, we further explore and analyse the interplay between an underlying Banach space $\X$ and $\overline{\D}$ being a spectral or complete spectral set for certain operators on $\X$ and $\ell_2(\X)$. Before going to the main theorem, we state a result from \cite{FAF} which will be useful.

\begin{thm}[\cite{FAF}, Corollary 2] \label{cor:new-001}
Let $\X$ be a complex Banach space. Then $\X$ is a Hilbert space if and only if for any two elements $x,y \in \X$ with $\|x\|=\|y\|$ the equality $\|ax + by\| = \|bx + ay\|$ holds for all real scalars $a,b$.
\end{thm}

Now we are in a position to present one of the main results of this paper.

\begin{thm}\label{thm:103}
Let $\X$ be a complex Banach space. Then the following are equivalent:

\begin{enumerate}
\item[(i)] $\X$ is a Hilbert space;

\smallskip

\item[(ii)] $\overline{\D}$ is a spectral set for $\widehat{M}_z$ on $\ell_2(\X)$, where $\widehat{M}_z(a_0, a_1, \dots)=(a_1,a_2, \dots)$;

\smallskip

\item[(iii)] $\overline{\D}$ is a spectral set for $M_z$ on $\ell_2(\X)$, where ${M}_z(a_0, a_1, \dots)=(\mathbf 0,a_0, \dots)$;

\smallskip

\item[(iv)] $\overline{\D}$ is a spectral set for every strict contraction on $\X$;

\smallskip

\item[(v)] $\overline{\D}$ is a complete spectral set for every contraction $T$ on $\X$ with $\|T\|=1$;

\smallskip

\item[(vi)] $\overline{\D}$ is a complete spectral set for the identity operator $I_\X$ on $\X$.
\end{enumerate}
\end{thm}

\begin{proof}
The forward implications $(i) \Rightarrow (ii)$, $(i) \Rightarrow (iii)$, $(i) \Rightarrow (iv)$ and $(i) \Rightarrow (v) \Rightarrow (vi)$ follow from Theorem \ref{thm:401}. We prove here $(ii) \Rightarrow (i)$, $(iii)\Rightarrow (i)$, $(iv)\Rightarrow (i)$ and $(vi) \Rightarrow (i)$.

\medskip

\noindent $(ii) \Rightarrow (i)$. For every $\alpha \in \D$, consider the automorphism $\phi_\alpha$ of the unit disk $\D$ defined by $\phi_\alpha(z) = \frac{z- \alpha}{1-\overline{\alpha}z}$. By the hypothesis, $\phi_\alpha(\widehat{M}_z)$ is a contraction on $\ell_2(\X)$. For any $x,y \in \X$ with $\|x\|=\|y\|$, define $\underline{x_0} = (\mathbf{0}, x, y, \mathbf{0}, \cdots)$ and set $\underline{y} = (I- \bar{\alpha}\widehat{M}_z)\underline{x_0}$, where $I$ is the identity operator on $\ell_2(\X)$. Then $\|\phi_\alpha(\widehat M_z)(\underline{y})\|\leq \|\underline{y}\|$ implies that
\begin{align}\label{eq:203}
   \|\widehat{M}_z\underline{x_0} - \alpha \underline{x_0}\|^2 \leq \|(I - \bar{\alpha}\widehat{M}_z)\underline{x_0}\|^2 &\Rightarrow \|(x,y- \alpha x, -\alpha y, \mathbf{0}, \cdots)\|^2 \leq \|( -\bar{\alpha}x, x-\bar{\alpha}y, y, \mathbf{0}, \cdots)\|^2 \nonumber \\
   & \Rightarrow \|x\|^2 + \|y-\alpha x\|^2 + \|\alpha y\|^2 \leq \|\bar{\alpha}x\|^2 + \|x -\bar{\alpha} y\|^2 + \|y\|^2 \nonumber \\
   & \Rightarrow \|y-\alpha x\| \leq \|x -\bar{\alpha} y\|, \quad \alpha \in \D \quad \left[\text{ as } \|x\|= \|y\|~ \right].
\end{align}
By interchanging the role of $x,y$ and replacing $\alpha$ by $\bar{\alpha}$ in (\ref{eq:203}), we have $\|x- \bar{\alpha} y\| \leq \|y - \alpha x\|$ for all $\alpha\in \D$. Consequently, for all $x,y$ with $\|x\|=\|y\|$ and $\alpha \in \D$, we have that $\|x-\alpha y\| = \|y - \alpha x\|$. In particular, for all $t\in (-1,1)$ and $x,y\in \X$ with $\|x\|=\|y\|$ we have $\|x-ty\|=\|y-tx\|$, which leads to the folowing identity.
\[
\|ax+by\| =\|ay+bx\|, \quad a, b \in \R, \quad \|x\|=\|y\|.
\]
Hence, $\X$ is a Hilbert space by Theorem \ref{cor:new-001}.

\medskip

\noindent $(iii)\Rightarrow (i)$. Let $x, y \in \X$ be arbitrary with $\|x\| = \|y\|$. Consider $\underline{x_0} = (x, y, \mathbf{0}, \cdots) \in \ell_2(\X)$ and set $\underline{y} = (I - \bar{\alpha}M_z)\underline{x_0}$ for all $\alpha \in \D$. An argument similar to that in the proof of $(ii) \Rightarrow (i)$ gives the rest.

\medskip

\noindent $(iv)\Rightarrow (i)$. Let $T\in\mathcal{B}(\X)$ be arbitrary with $\|T\|=1$. For every $\alpha\in \D$, consider the automorphism $\phi_\alpha \in \Aut(\D)$ defined by $\phi_\alpha(z)= \frac{z-\alpha}{1-\overline{\alpha}z}$. By hypothesis, $\overline{\D}$ is a spectral set for $rT$ for all $r\in [0,1)$. So, $(rT - \alpha I_\X)(I_\X - \bar{\alpha} rT)^{-1}$ is a contraction on $\X$, i.e., $\|(rT-\alpha I_\X) (I_\X - \bar{\alpha}rT)^{-1}x\| \leq \|x\|$ for any $x\in \mathbb{X}$. By substituting $x= (I_\X- \bar{\alpha}rT)y$, we have
\begin{equation}\label{eq:204}
   \|(rT-\alpha I_\X)y\| \leq \|(I_\X - \bar{\alpha}rT)y\|, \quad y\in \X, \quad \alpha\in \D, \quad r\in [0,1).
\end{equation}
Now, taking $r\to 1$ in (\ref{eq:204}), we have
\begin{equation}\label{eq:205}
  \|(T-\alpha I_\X)y\| \leq \|(I_\X - \bar{\alpha}T)y\|, \quad y\in \X, \quad \alpha\in \D, 
\end{equation}
which holds for all $T\in \mathcal{B}(\X)$ with $\|T\|=1$. The rest of the proof after the inequality (\ref{eq:205}) follows from Foias's proof of Theorem \ref{thm:Foias-New}. Also, in \cite[Theorem 1.9]{GPII} Pisier provided a similar proof to the fact that $\X$ is a Hilbert space if and only if (\ref{eq:205}) holds for all $T\in\mathcal{B}(\X)$ with $\|T\|=1$. For the convenience of the readers, we imitate their techniques to finish the rest of the proof here. Let $p, q \in \X$ be two arbitrary vectors such that $\|p\|=\|q\| =1$. By the Hahn Banach extension theorem, there exists a bounded linear functional $f_p: \X \to \C$ such that $\|f_p\| = \|p\| =1$ and $f_p(p) = \|p\|^2 =1$. Consider the operator $T: \X \to \X$ defined by $T(x) = f_p(x)q$ for all $x\in \X$. Then we have $T(p)=q$. Since $\|T\| = \|f_p\| \|q\| =1$, by substituting $y=p$ in (\ref{eq:205}) we have
\[
  \|q-\alpha p\| \leq \|p - \bar{\alpha} q\|, \quad \alpha \in \D.
\]
Interchanging the roles of $p$ and $q$ in the above inequality gives $\|p-\bar{\alpha} q \| \leq \|q-\alpha p\|$. Consequently, we have 
\[
   \|p- \bar{\alpha} q\| = \|q- \alpha p\|, \quad \alpha \in \D.
\]
The continuity of the norm implies that $\|p+ tq \| = \|q + tp\|$ holds for any $t\in [-1, 1]$. Moreover, for $t\in \R$ with $|t| \geq 1$, we have \[
\|p + tq\| = t\|\frac{p}{t} + q\| = t \|p + \frac{1}{t} q\| = \|tp + q\|.
\]
It follows from here that $\|p + tq\|= \|tp + q\|$ for all $t\in \mathbb R$. This further implies that
\[
  \|ap + bq \| = \|bp + aq\|, \quad a, b \in \R.
\]
Therefore, we conclude that if $p, q\in \X$ with $\|p\|=\|q\|=1$, then $\|ap + bq\| = \|bp + aq\|$ holds for all $a,b \in \R$. Now, for arbibrary $x, y \in \X$ with $\|x\| = \|y\| \neq 0$, and $a, b \in \R$ we have 
\[
  \|ax + by \| = \|x\| \left\|a\frac{x}{\|x\|} + b \frac{y}{\|y\|}\right\| = \|x\| \left\|b \frac{x}{\|x\|} + a\frac{y}{\|y\|}\right\| = \|ay + bx\|.
\]
The first and second equalities hold because of the fact that $\|x\|=\|y\|\neq 0$. Of course, $\|ax+by\| = \|bx + ay\|$ holds for all $a, b\in \R$ if either $x$ or $y$ or both of them are zero vectors. Hence, $\X$ is a Hilbert space by Theorem \ref{cor:new-001}. 

\smallskip

\noindent $(vi) \Rightarrow (i)$. Consider the matricial polynomial $F$ defined by $F(z)= \left(\begin{matrix}
z\slash \sqrt{2} & - z\slash \sqrt{2} \\
z\slash \sqrt{2} & z\slash \sqrt{2}
\end{matrix}\right)$. Since the scalar matrix $\left(\begin{matrix}
1\slash \sqrt{2} & - 1\slash \sqrt{2} \\
1\slash \sqrt{2} & 1\slash \sqrt{2}
\end{matrix}\right)= \left(\begin{matrix}
\cos \frac{\pi}{4} & - \sin \frac{\pi}{4} \\
\sin \frac{\pi}{4} & \, \cos \frac{\pi}{4}
\end{matrix}\right)$ is a unitary, it follows that $\|F(z)\|=|z|$ for every $z\in \overline{\D}$ and hence $\|F\|_{\infty, \overline{\D}} =1$. So, by assumption $F(I_\X) = \begin{bmatrix}
I_\X\slash \sqrt{2} & - I_\X\slash \sqrt{2} \\
I_\X\slash \sqrt{2} & I_\X\slash \sqrt{2}
\end{bmatrix}$ is a contraction on $\X \oplus_2 \X$. Let $(x,y) \in \X \oplus_2 \X$ be arbitrary. Then $\|F(I_\X)(x,y)\|^2 \leq \|(x,y)\|^2$ and this implies that
\begin{align}
& \quad \left\| \left( \frac{x-y}{\sqrt{2}}, \frac{x+y}{\sqrt{2}}\right)\right\|^2 \leq \|x\|^2 + \|y\|^2 \notag \\
& \Rightarrow \|x-y\|^2 + \|x+y\|^2 \leq 2 (\|x\|^2 + \|y\|^2). \label{eq:401}
\end{align}
Set $x' = \dfrac{x+y}{2}$ and $y' = \dfrac{x-y}{2}$. Since (\ref{eq:401}) holds for all $(x,y) \in \X \oplus_2 \X$, substituting $(x',y')$ in (\ref{eq:401}) we have
\begin{align}
& \quad  \|x'-y'\|^2 + \|x'+y'\|^2 \leq 2 (\|x'\|^2 + \|y'\|^2) \notag \\ &  \Rightarrow 2(\|x\|^2 + \|y\|^2) \leq \|x-y\|^2 + \|x+y\|^2. \label{eq:402}
\end{align}
The inequalities $(\ref{eq:401})$ and $(\ref{eq:402})$ together establish the parallelogram law on $\X$. Hence $\X$ is a Hilbert space and the proof is complete.
\end{proof}

\smallskip

\section{Spectral set vs. complete spectral set vs. dilation of a Banach space contraction}\label{Subsec:041}

\vspace{0.2cm}

\noindent Recall that Theorem \ref{thm:401} provides equivalence of the following four statements for a Hilbert space operator $T$: $T$ is a contraction, i.e. $\|T\| \leq 1$; the closed unit disk $\overline{\D}$ is a spectral set for $T$; the closed unit disk $\overline{\D}$ is a complete spectral set for $T$; $T$ admits dilation to a Hilbert space isometry. Theorem \ref{thm:Foias-New} due to Foias guarantees that $\overline{\D}$ may not be a spectral set for a Banach space contraction. In Section \ref{Sec:03}, we have seen that even more is true. Indeed, for every $r> \frac{1}{3}$, there is a Banach space contraction $T_r$ with $\|T_r\|=r$ such that $\overline{\D}$ is not a spectral set for $T_r$. Thus, the operator-norm versus spectral set issue is resolved for Banach space operators. In this Section, we show by counter examples that any two among the other conditions of Theorem \ref{thm:401} are not equivalent in Banach space setting.

\subsection{Spectral set vs. dilation in Banach space}\label{Subsubsec:0411}
We begin with a useful proposition that shows an interplay between operator norm and dilation in Banach space.

\begin{prop}\label{prop:4111}
For every $\lambda\in (0,1)$, there is a Banach space contraction $T_\lambda$ with $\|T_{\lambda}\|=\lambda$ such that $T_\lambda$ does not dilate to any isometry.
\end{prop}

\begin{proof}
Let $\X = (\C^2, \|\cdot\|_1)$. For every $\lambda\in (0,1)$, consider the operator $T_\lambda: \X \to \X$ defined by $T_\lambda(x,y) = \lambda(x+y, 0)$. Evidently, $\|T_\lambda\|= \lambda$. In view of Theorem \ref{thm:102}, it suffices to show that the function $A_{T_\lambda}(x) = \left(\|x\|^2 - \|T_\lambda x\|^2 \right)^{1\slash 2}$ does not define a norm on $\X$ in order to prove that $T_\lambda$ does not dilate to any isometry.  Let us suppose the contrary, i.e., let $A_{T_\lambda}$ define a norm on $\X$. Let $e_1 = (1,0)$ and $e_2 = (0,1)$. Then the triangle inequality implies that
\[
   A_{T_\lambda}(e_1 - e_2) = 2 \leq A_{T_\lambda}(e_1) + A_{T_\lambda}(e_2) = 2 \sqrt{1- \lambda^2},
\]
which is a contradiction to the fact that $\lambda\in (0,1)$. Hence $T_{\lambda}$ does not dilate to any isometry.
\end{proof}

Given below an example of a Banach space contraction $T$ such that $\overline{\D}$ is a spectral set of $T$ but $T$ does not dilate to any isometry.

\begin{eg}\label{eg:4111}
Let $\X = (\C^2, \|\cdot\|_1)$. Then, for every $\lambda \in (0, \frac{1}{3}]$, Proposition \ref{prop:4111} shows the existence of a contraction $T_{\lambda}$ on $\X$ with $\|T_{\lambda}\|=\lambda$ such that $T_{\lambda}$ does not dilate to any isometry. Again, Theorem \ref{thm:302} shows that $\overline{\D}$ is a spectral set for each such $T_{\lambda}$.
\end{eg}
We now give an example to show the other way, that is, an example of  a Banach space contraction $T$ that dilates to an isometry but $\overline{\D}$ is not a spectral set for $T$.
 
\begin{eg}\label{eg:4112}
Let $\X$ be a non-Hilbert Banach space and consider the backward shift operator $\widehat M_z$ on $\ell_2(\X)$, i.e., $\widehat M_z(a_0, a_1, \dots)=(a_1,a_2, \dots)$. Set $S_\lambda = \lambda \widehat{M}_z$ for every $\lambda\in (0,1)$. First we show that $S_{\lambda}$ dilates to an isometry for every $\lambda\in (0,1)$. By Theorem \ref{thm:102}, it suffices to show that $A_{S_\lambda}$ defines a norm on $\ell_2(\X)$. The positivity of $A_{S_\lambda}$ follows from the fact that $S_\lambda$ is a strict contraction. Indeed, $\|S_\lambda\| = \lambda < 1$. The homogeneity of $A_{S_\lambda}$ follows from the linearity of $S_\lambda$. To show the triangle inequality, consider the operator $S_\mu: \ell_2(\X) \to \ell_2(\X)$ defined by $S_\mu(\{x_n\}) = \{\mu_n x_n \}$, where $\mu_1 = 1$ and $\mu_n = \sqrt{1-\lambda^2}$ for all $n\geq 2$. It is easy to see that $S_\mu$ is a well defined bounded linear map on $\ell_2(\X)$ and $\|S_\mu\| = 1$. Also, for all $\{x_n\}\in \ell_2(\X)$ we have

\[
   A_{S_\lambda}(\{x_n\}) = \left( \sum_{n=1}^\infty \|x_n\|^2 - \sum_{n=2}^\infty \lambda^2 \|x_n\|^2 \right)^{1\slash 2} = \left( \|x_1\|^2 + \sum_{n=2}^\infty \mu_n^2 \|x_n\|^2 \right)^{1\slash 2} = \|S_\mu (\{x_n\})\|.
\]

Consequently, the triangle inequality for $A_{S_\lambda}$ follows from the linearity of $S_\mu$. Thus, $A_{S_{\lambda}}$ defines a norm on $\X$ and consequently $S_{\lambda}$ dilates to an isometry for every $\lambda \in (0,1)$.

\smallskip

Now, we show that there exists $\lambda \in (1\slash 3, 1)$ such that $\overline{\D}$ is not a spectral set of $S_\lambda$. Note that it follows from Theorem \ref{thm:302} that $\overline{\D}$ is a spectral set for $S_{\lambda}$ when $\lambda \in (0,1\slash 3]$. If possible let, $\overline{\D}$ is  a spectral set of $S_\lambda$ for all $\lambda\in (1\slash 3 ,1)$. For every $\alpha\in \D$, consider $\phi_\alpha\in \Aut(\D)$ defined by $\phi_\alpha(z) = \frac{z- \alpha}{1-\bar{\alpha}z}$. Then $\phi_\alpha(S_\lambda)$ is a contraction on $\X$ for all $\lambda\in (1\slash 3, 1)$. Let $\underline{x}\in \ell_2(\X)$ be arbitrary and set $\underline{y}= (I-\bar{\alpha}S_\lambda)\underline{x}$, where $I$ is the identity operator on $\ell_2(\X)$. Then $\|\phi_\alpha(S_\lambda)(\underline{y})\|\leq \|\underline{y}\|$ implies that
\begin{equation}\label{eq:4111}
   \|(S_\lambda - \alpha I) \underline{x}\| \leq \|(I - \bar{\alpha}S_\lambda)\underline{x}\|, \quad 1\slash 3 < \lambda < 1.
\end{equation}
Now, substituting $S_\lambda = \lambda \widehat{M}_z$ and taking limit as $\lambda\to 1$ in (\ref{eq:4111}) we have
\begin{equation} \label{eqn:411A}
   \| (\widehat{M}_z - \alpha I)\underline{x} \| \leq \|(I-\bar{\alpha}\widehat{M}_z)\underline{x}\|, \quad \underline{x}\in \ell_2(\X).
\end{equation}
Now, following the proof of $(ii) \Rightarrow (i)$ of Theorem \ref{thm:103}, we see that the inequality (\ref{eqn:411A}) turns $\X$ into a Hilbert space, which is a contradiction as we started with a non-Hilbert Banach space $\X$. Consequently, there exists $\lambda\in (1\slash 3, 1)$ such that $\overline{\D}$ is not a spectral set of $\lambda \widehat{M}_z$ on $\ell_2(\X)$, though $\lambda \widehat{M}_z$ dilates to an isometry for every $\lambda \in (0,1)$.
\end{eg}

\subsection{Spectral set vs. complete spectral set for Banach space contractions}\label{subsubsec:0413}

It is evident from the definition of spectral set and complete spectral set that if a compact set $K$ is a complete spectral set for a Banach space operator $T$ then it is also a spectral set for $T$. However, the fact that the closed unit disk $\overline{\D}$ is a spectral set for an operator on a Banach space does not guarantee that $\overline{\D}$ is also a complete spectral set for it. Consider any complex non-Hilbert Banach space $\X$. Then for any $f\in Rat(\overline{\D})$, we have $\|f(I_\X)\| = |f(1)|\leq \|f\|_{\infty, \overline{\D}}$. This shows that $\overline{\D}$ is a spectral set of the identity operator $I_\X$ on $\X$. The equivalence of conditions $(i)$ and $(vi)$ of Theorem \ref{thm:103} clearly shows that $\overline{\D}$ is not a complete spectral set of $I_\X$.

\subsection{Dilation vs. complete spectral set for Banach space operators}\label{Subsubsec:0412}

As we discussed in the previous Subsection, complete spectral set implies spectral set for a Banach space operator. In Example \ref{eg:4112}, we have shown the existence of a Banach space contraction $S_{\lambda}$ that dilates to an isometry but $\overline{\D}$ is not a spectral set for $S_{\lambda}$. So, naturally $\overline{\D}$ is not a complete spectral set for that $S_{\lambda}$ too. Here we construct an example of a Banach space contraction $T$ that has $\overline{\D}$ as a complete spectral set but $T$ does not dilate to any isometry.
\begin{eg}
Let $\HS$ be a Hilbert space. Consider the Banach space $\X = \HS \oplus_1 \HS$ that consists of elements of the form $(h_1, h_2)$ with $h_1, h_2\in \HS$ and is equipped with the norm $\|(h_1, h_2)\| = \|h_1\| + \|h_2\|$. For each $r\in (0,1)$, consider the operator $T_r: \X \to \X$ defined by $T_r(h_1,h_2 ) = (rh_1, \mathbf{0})$. First we show that $\overline{\D}$ is a complete spectral set for $T_r$ for any $r\in (0,1)$. Consider an arbitrary matricial rational function $F=[f_{ij}]_{n\times n}$ from $Rat_n(\overline{\D})$ for any $n\in \mathbb N$. Let $(\underline{x_1}, \underline{x_2}, \dots, \underline{x_n})\in \underbrace{\X \oplus_2 \X \oplus_2 \dots \oplus_2 \X}_{n-times}$ be arbitrary with $\underline{x_i} = (x_{i1}, x_{i2}) \in \HS \oplus_1 \HS =\X$ for $1\leq i \leq n$. Note that $f_{ij}(T_r)\underline{x_j} =( f_{ij}(rI_\HS)x_{j1}, \mathbf{0})$ for all $1\leq i,j \leq n$. So, we have
\begin{align*}
 \left\| F(T_r)(\underline{x_1}, \underline{x_2}, \dots, \underline{x_n}) \right\|  = \left( \sum_{i=1}^n \left\|\sum_{j=1}^n f_{ij}(T_r)\underline{x_j} \right\|^2 \right)^{1\slash 2}  & = \left( \sum_{i=1}^n \left\|\sum_{j=1}^n \left(f_{ij}(rI_\HS)\underline{x_{j1}}, \mathbf{0}\right) \right\|^2 \right)^{1\slash 2} \\
 & = \left\|F(rI_\HS)(x_{11}, x_{21}, \dots, x_{n1})\right\| \\
& \leq \|F\|_{\infty, \overline{\D}} \left(\sum_{i=1}^n \|x_{i1}\|^2 \right)^{1\slash 2} \\
& \leq \|F\|_{\infty, \overline{\D}} \left(\sum_{i=1}\|\underline{x_i}\|^2 \right)^{1\slash 2},
\end{align*}
where the second last inequality follows from Theorem \ref{thm:401}, since $rI_{\HS}$ is a contraction on the Hilbert space $\HS$ for any $r\in (0,1)$. So, we have that $\|F(T_r)\| \leq \|F\|_{\infty, \overline{\D}}$ and consequently $\overline{\D}$ is a complete spectral set of $T_r$ for any $r\in (0,1)$.

\smallskip

Now, we show that there exists $r\in (0,1)$ such that $T_r$ does not dilate to any isometry. By virtue of Theorem \ref{thm:102}, it suffices to show that $A_{T_r}$ does not define a norm on $\X$ for some $r\in (0,1)$. Suppose the contrary, i.e.,  assume that $A_{T_r}$ defines a norm on $\X$ for all $r\in (0,1)$. Let $e\in \HS$ be a unit vector. Consider $e_1= (e, \mathbf{0})$ and $e_2 = (\mathbf{0}, e)$ in $\X$. Then we have $\|e_i\|= 1$ for $i=1,2$ and $\|(e_1+e_2)\| = 2$. Note that $T_r(e_1) =( re,\mathbf{0})$ and $T_r(e_2) = (\mathbf{0}, \mathbf{0})$. So, we have $A_{T_r}(e_1) = \sqrt{1- r^2}$, $A_{T_r}(e_2) = 1$ and $A_{T_r}(e_1 + e_2) = \sqrt{4- r^2 }$. Thus, the triangle inequality implies that
\begin{equation}\label{eq:4121}
   \sqrt{4- r^2 } = A_{T_r}(e_1+e_2) \leq A_{T_r}(e_1) + A_{T_r}(e_2) = 1 + \sqrt{1- r^2}, \quad 0 < r < 1.
\end{equation}
Now, taking limit as $r\to 1$ in (\ref{eq:4121}) we have $\sqrt{3} \leq 1$, which is a contradiction. Hence, there exists $\widetilde r\in (0,1)$ such that $T_{\widetilde r}$ does not dilate to any isometry. However, we already have proved that $\overline{\D}$ is  a complete spectral set for $T_{\widetilde r}$.
\end{eg}

\vspace{0.1cm}

\noindent \textbf{Acknowledgement.} The first named author is supported by the ``Prime Minister's Research Fellowship (PMRF)" of Govt. of India with Award No. PMRF-1302045. The second named author is supported by ``Core Research Grant" of Science and Engineering Research Board (SERB), Govt. of India, with Grant No. CRG/2023/005223 and the ``Early Research Achiever Award Grant" of IIT Bombay with Grant No. RI/0220-10001427-001.

\end{document}